\documentclass[11pt,a4paper]{article}
\usepackage{epsfig}
\usepackage[numbers]{natbib}
\usepackage{amsmath, amsthm, amssymb}
\usepackage{epsfig}
\usepackage[dvipdfm]{hyperref}
\usepackage[margin=3cm, top=3.5cm, bottom=3.5cm]{geometry}

\newcommand\blfootnote[1]{%
  \begingroup
  \renewcommand\thefootnote{}\footnote{#1}%
  \addtocounter{footnote}{-1}%
  \endgroup
}

\newtheorem{theorem}{Theorem}[section]

\newtheorem{proposition}{Proposition}[section]

\newtheorem{example}{Example}[section]

\everymath{\displaystyle}
\allowdisplaybreaks

\title{Sharp Hardy inequalities in an exterior of a ball}
\date{}

\author{
Nikolai Kutev\thanks{Institute of Mathematics and Informatics, Bulgarian Academe of Sciences, 1113, Sofia, Bulgaria}
 \and Tsviatko Rangelov
 \footnotemark[1]
}

\begin{document}

\maketitle
\blfootnote{Corresponding author: T. Rangelov, rangelov@math.bas.bg}

\begin{abstract}
\noindent
New Hardy type inequalities in sectorial area and as a limit in an exterior of a ball are proved. Sharpness of  the inequalities is shown as well.
\end{abstract}

\vspace{2pt}


\vspace{2pt}

\section{Introduction}
\label{sec1}
The classical Hardy inequality, proved in \citet{Ha20,
Ha25} states
\begin{equation}
\label{1eq1} \int_0^\infty |u'(x)|^px^\alpha dx\geq
\left(\frac{p-1-\alpha}{p}\right)^p\int_0^\infty
x^{-p+\alpha}|u(x)|^pdx
\end{equation}
where $1<p<\infty$, $\alpha<p-1$ and $u(x)$ is absolutely continuous
on $[0,\infty)$, $u(0)=0$.

There are several generalizations of \eqref{1eq1} in multidimensional case mostly in bounded domains, see \citet{Da98, GM13, BEL15, KR20} and the literature therein. For unbounded domains, in the exterior of a boll, only few generalizations of \eqref{1eq1} are reported in \citet{WZ03, AFT09, AL10}.

For example, in Theorem 1.1 in \citet{WZ03} for all function $u\in D_a^{1,2}(R^n)$, where $D_a^{1,2}(R^n)$ is the weighted Sobolev space - the comletion of $C_0^\infty(R^n)$ with the norm $\int_{R^n}|x|^{-2a}|u|^2dx$, the following inequality for $a<\frac{n-2}{2}$ is proved
\begin{equation}
\label{eq01}
\int_{B_1^c}|x|^{-2a}|\nabla u|^2dx\geq\left(\frac{n-2-2a}{2}\right)^2\int_{B_1^c}\frac{|u|^2}{|x|^2(a+1)}dx+\frac{n-2-2a}{2}\int_{\partial B_1^c}u^2,
\end{equation}
where $B_1^c=\{|x|>1\}$.

Let us mention that in \citet{WZ03} Hardy inequalities with weights in unbounded domains $\Omega\subset R^n, 0\notin \partial\Omega$ are also considered, see Theorem 1.3 and Remark 1.5.

In \citet{AFT09}, Corollary, for $n\geq3$ and $u\in C_0^\infty(B_1^c)$ the limiting case of the Caffarelli--Kohn--Nirenberg inequality
\begin{equation}
\label{eq03}
\int_{B_1^c}|\nabla u|^2dx\geq\left(\frac{n-2}{2}\right)^2\int_{B_1^c}\frac{u^2}{|x|^2}dx+C_n(a)\left[\int_{B_1^c}X_1^{\frac{2(n-1)}{n-2}}\left(a,\frac{1}{|x|}\right)u^{\frac{2n}{n-2}}dx\right]^{\frac{n-2}{n}},
\end{equation}
is proved, where $X_1(a,s)=a-\ln s)^{-1}$, $a>0, 0<s\leq1$ and the constant $C_n(a)$ is given explicitly.

Finally, in Corollary 1 and Remark 1 in \citet{AL10} the following Hardy inequality is proved for $n\geq3$ and $u\in W_0^{1,2}(B_r^c)$
\begin{equation}
\label{eq04}
\int_{B_r^c}|\nabla u|^2dx\geq\left(\frac{n-2}{2}\right)^2\int_{B_r^c}\frac{u^2}{|x|^2}dx+\frac{1}{4}\int_{B_r^c}\left(\frac{1}{|x-r|^2}-\frac{1}{|x|^2}\right)u^2dx.
\end{equation}

At the end of the paper we compare the inequalities  \eqref{eq01}, \eqref{eq03}, \eqref{eq04} with our results.

The aim of the present work is to derive new Hardy inequalities in the exterior of a ball. There are listed also functions for which these inequalities are sharp, i.e., inequalities with an optimal constant of the leading term become equations.

\section{Inequalities in sectorial area}
\label{sec2}
We start with Hardy inequalities in sectorial area $B_R\backslash B_r$ where
$B_R$, $B_r$ are balls centered at zero, $0<r<R$ .  Let $1<p$, $p'=\frac{p}{p-1}$, $2\leq n$ and
denote
$m=\frac{p-n}{p-1}=\frac{p-n}{p}p'$.

For  functions $u$  such that $\int_{B_R\backslash B_r}\left|\frac{\langle x,\nabla
u\rangle}{|x|}\right|^pdx<\infty$ let us define two sets
$$
\begin{array}{lll}
M_1(r,R)&=&\left\{\begin{array}{l}
\left|\frac{R^m-\hat{R}^m}{m}\right|^{1-p}\int_{\partial
B_{\hat{R}}}|u|^pd\sigma\rightarrow0, \ \ \hat{R}\rightarrow R-0, \
\ m\neq0,
\\[1pt]
\\
\left|\ln{\frac{R}{\hat{R}}}\right|^{1-n}\int_{\partial
B_{\hat{R}}}|u|^nd\sigma\rightarrow0, \ \ \hat{R}\rightarrow R-0, \
\ m=0
\end{array}\right.
\\[1pt]
\\
M_2(r,R)&=&\left\{\begin{array}{l}
\left|\frac{\hat{r}^m-r^m}{m}\right|^{1-p}\int_{\partial
B_{\hat{r}}}|u|^pd\sigma\rightarrow0, \ \ \hat{r}\rightarrow r+0, \
\ m\neq0,
\\[1pt]
\\
\left|\ln{\frac{\hat{r}}{r}}\right|^{1-n}\int_{\partial
B_{\hat{r}}}|u|^nd\sigma\rightarrow0, \ \ \hat{r}\rightarrow r+0, \
\ m=0
\end{array}\right.
\end{array}
$$
where $\langle , \rangle$ is a scalar product in $R^n$.

Let functions $\psi_j$ be solutions of the problems:
$$
\begin{array}{lll}
&& -\Delta_p\psi_1=0, \ \ \hbox{ in } B_R\backslash \bar{B}_r,\ \
\psi_1|_{\partial B_R}=0,\ \ \psi_1|_{\partial B_r}=1,
\\[1pt]
\\
&& -\Delta_p\psi_2=0, \ \ \hbox{ in } B_R\backslash \bar{B}_r,\ \
\psi_2 |_{\partial B_R}=1,\ \ \psi_2|_{\partial B_r}=0.
\end{array}
$$
Their explicit form is:
$$
\psi_1(x)=\left\{\begin{array}{l}\frac{R^m-|x|^m}{R^m-r^m},
\ \ m\neq 0, \\
\frac{\ln{\frac{R}{|x|}}}{\ln{\frac{R}{r}}},
\ \ m= 0\end{array}\right. , \ \
\psi_2(x)=\left\{\begin{array}{l}\frac{|x|^m-r^m}{R^m-r^m},
\ \ m\neq 0, \\
\frac{\ln{\frac{|x|}{r}}}{\ln{\displaystyle\frac{R}{r}}},
\ \ m= 0\end{array}\right..
$$

We can define  vector functions $f_i$  as
$f_i=\left|\displaystyle\frac{\nabla\psi_i}{\psi_i}\right|^{p-2}\displaystyle\frac{\nabla\psi_i}{\psi_i}$.
in $B_R\backslash \bar{B}_r$
and
\begin{equation}
\label{eq7}\begin{array}{lll}
L^i(u)&=&\int_{B_R\backslash
B_r}\left|\frac{\langle \nabla\psi_i,\nabla u\rangle}{|\nabla\psi_i|}\right|^pdx=\int_{B_R\backslash B_r}\left|\frac{\langle x,\nabla
u\rangle}{|x|}\right|^pdx,
\\[1pt]
\\
K^i(u)&=& \int_{B_R\backslash B_r}\left|\frac{\nabla
\psi_i}{\psi_i}\right|^p|u|^pdx=\int_{B_R\backslash B_r}|f_i|^{p'}|u|^pdx ,
\\[1pt]
\\
K^i_0(u)&=& \int_{\partial(B_R\backslash
B_r)}\left|\frac{\nabla
\psi_i}{\psi_i}\right|^{p-2}\left\langle\frac{\nabla\psi_i}{\psi_i},\nu\right\rangle|u|^pds=\int_{\partial(B_R\backslash
B_r)}\langle f_i,\nu\rangle|u|^pdx.
\end{array}
\end{equation}
Here $\nu$ is the outward normal to $B_R\backslash B_r$.

The following  Theorem takes place.
\begin{theorem}
\label{th1} For every $u\in M_i(r,R)$ we have
\begin{equation}
\label{eq6}
L^i(u)\geq\left(\frac{1}{p}\right)^p\frac{[K^i_0(u)+(p-1)K_1^i(u)]^p}{(K_1^i(u))^{p-1}}=K^i(u).
\end{equation}
where $\nu$ is the outward normal to $B_R\backslash B_r$.
\end{theorem}

\begin{proof}
We follow the proof of Proposition 1 in \citet{FKR14c}. Since
\begin{equation}
\label{eq8} \int_{B_{\hat{R}}\backslash
B_{\hat{r}}}\langle f_i,\nabla |u|^p\rangle dx=p\int_{B_{\hat{R}}\backslash
B_{\hat{r}}}|u|^{p-2}u\langle f_i,\nabla u\rangle dx ,
\end{equation}
where $r<\hat{r}<\hat{R}<R$. Then applying H\"older inequality on
the \textit{rhs} of \eqref{eq8} with $\displaystyle \frac{\langle x,\nabla
u\rangle}{|x|}$ and $\displaystyle |f_i||u|^{p-2}u$ as factors of the
integrand we get
\begin{equation}
\label{eq9}
\begin{array}{lll}\int_{B_{\hat{R}}\backslash
B_{\hat{r}}}\langle f_i,\nabla|u|^p\rangle dx &\leq&
p\left(\int_{B_{\hat{R}}\backslash
B_{\hat{r}}}\left|\frac{\langle x,\nabla u\rangle}{|x|}\right|^pdx\right)^{1/p}
\\ &\times& \left(\int_{B_{\hat{R}}\backslash
B_{\hat{r}}}|f_i|^{p'}|u|^pdx\right)^{1/p'}.
\end{array}
\end{equation}

Rising to $p$ power both sides of \eqref{eq9} it follows that
\begin{equation}
\label{eq10} \int_{B_{\hat{R}}\backslash B_{\hat{r}}}
\left|\frac{\langle x,\nabla
u\rangle}{|x|}\right|^pdx\geq\frac{\left|\frac{1}{p}\int_{B_{\hat{R}}\backslash
B_{\hat{r}}}\langle f_i,\nabla|u|^p\rangle dx\right|^p}{\left(
\int_{B_{\hat{R}}\backslash
B_{\hat{r}}}|f_i|^{p'}|u|^pdx\right)^{p-1}}.
\end{equation}

Integrating by parts the numerator of the right hand side of \eqref{eq10} we
get
\begin{equation}
\label{eq11}
\begin{array}{lll}
&& \frac{1}{p}\int_{B_{\hat{R}}\backslash
B_{\hat{r}}}\langle f_i,\nabla|u|^p\rangle dx=\frac{1}{p} \int_{\partial
B_{\hat{R}}\cup
\partial B_{\hat{r}}}\langle f_i,\nu\rangle |u|^pdS-\frac{1}{p}\int_{B_{\hat{R}}\backslash
B_{\hat{r}}}\hbox{div} f_i|u|^pdx
\\
&&=\frac{1}{p}\int_{\partial B_{\hat{R}}\cup
\partial B_{\hat{r}}}\langle f_i,\nu\rangle |u|^pdS
+\left(\frac{p-1}{p}\right)\int_{B_{\hat{R}}\backslash
B_{\hat{r}}}|f|^{p'}|u|^pdx
\\ 
&&\rightarrow\displaystyle \frac{1}{p}((p-1)K^i+K^i_0), \hbox{ when
} \hat{R}\rightarrow R-0,\ \ \hat{r}\rightarrow r+0.
\end{array}
\end{equation}
 Note that
$\int_{\partial B_R\cup\partial
B_r}\langle f_i,\nu\rangle |u|^pdS\geq 0$ for $u\in M_i(r,R)$, since
$\nu|_{\partial B_R}=\frac{x}{|x|}|_{\partial B_R}$,
$\nu|_{\partial B_r}=-\frac{x}{|x|}|_{\partial B_r}$.
From \eqref{eq10} and \eqref{eq11} we obtain \eqref{eq6} since
$$
-\hbox{div}f_i=(p-1)|f_i|^{p'}.
$$
\end{proof}
\section{Inequalities in an exterior of a ball  $B_r^c=R^n\backslash \bar{B}_r$}
\label{sec3}
Let us introduce functions $u$ such that 
$$
\int_{B_r^c}\left|\frac{<x,\nabla
u>}{|x|}\right|^pdx<\infty, \ \  \int_{B_r^c}\frac{
|u|^p}{|x|^{(n-1)p'}}dx<\infty
$$
and define two sets
$$
\begin{array}{lll}
M_1(r,\infty)&=&\left\{\begin{array}{l}
R^{1-n}\displaystyle\int_{\partial
B_{R}}|u|^pd\sigma\rightarrow0, \ \ R\rightarrow \infty, \
\ m\geq0,
\\
R^{1-p}\int_{\partial
B_{R}}|u|^pd\sigma\rightarrow0, \ \ R\rightarrow \infty, \
\ m<0. \end{array}\right.
\\[2pt]
\\
M_2(r,\infty)&=&\left\{\begin{array}{l}
\left|\frac{\hat{r}^m-r^m}{m}\right|^{1-p}\int_{\partial
B_{\hat{r}}}|u|^pd\sigma\rightarrow0, \ \ \hat{r}\rightarrow r+0, \ \
m\neq0,
\\
\left(\ln\frac{\hat{r}}{r}\right)^{1-p}\int_{\partial
B_{\hat{r}}}|u|^pd\sigma\rightarrow0, \ \ \hat{r}\rightarrow r+0, \ \
m=0.
\end{array}\right.
\end{array}
$$
In a similar way we can prove Theorem \ref{th1}, replacing
$B_R\backslash \bar{B}_r$ with $B^c_r=R^n\backslash \bar{B}_r$ and $\partial(B_R\backslash \bar{B}_r)$
with $\partial B^c_r=\partial B_r$ and $L^i(u), K^I_0(u), K^i_1(u)$ define in \eqref{eq7} for $R\rightarrow\infty$. The inequalities below for functions of $M_i(r,\infty), i=1,2$ can be obtained with a limit $R\rightarrow\infty$.
\begin{proposition}
\label{prop1}
For every $u\in M_1(r,\infty)$ the following inequalities hold:
\begin{itemize}
\item[(i)]
\begin{equation}
\label{eq21}\begin{array}{lll}
&&\left(\int_{
B^c_r}\frac{|u|^p}{|x|^{(n-1)p'}}dx\right)^{\frac{1}{p'}}
\left(\int_{B^c_r}\left|\frac{<x,\nabla
u>}{|x|}\right|^pdx\right)^{\frac{1}{p}}
\\[1pt]
\\
&&\geq\frac{1}{p}r^{1-n}\int_{\partial B_r}|u|^pdS, \ \ for \ \ m>0.
\end{array}
\end{equation}
With function $u_\alpha(x)=e^{-\alpha|x|^m}, \ \ \alpha>0$ inequality
\eqref{eq21} becomes equality.
\item[(ii)] For $m<0$  we get:
\begin{equation}
\label{eq22}\begin{array}{lll}
&&\left(\int_{B^c_r}\left|\frac{<x,\nabla
u>}{|x|}\right|^pdx\right)^{\frac{1}{p}}
\geq\frac{|m|}{p'}\left(\int_{
B^c_r}\frac{|u|^p}{|x|^{p}}dx\right)^{\frac{1}{p}}
\\[1pt]
\\
&&+\frac{1}{p}r^{1-p}\int_{\partial
B_r}|u|^pdS\left(\int_{
B^c_r}\frac{|u|^p}{|x|^{p}}dx\right)^{-\frac{1}{p'}}, \ \ for \ \  m<0.
\end{array}
\end{equation}
With function $u_k(x)=|x|^{k m}, \ \ k>p'$ inequality
\eqref{eq22} becomes an equality.
\item[(iii)]
\begin{equation}
\label{eq23}\begin{array}{lll}
&&\left(\int_{
B^c_r}\frac{|u|^n}{|x|^{n}}dx\right)^{\frac{n-1}{n}}
\left(\int_{B^c_r}\left|\frac{<x,\nabla
u>}{|x|}\right|^ndx\right)^{\frac{1}{n}}
\\[1pt]
\\
&&\geq\frac{1}{n}r^{1-n}\int_{\partial B_r}|u|^ndS, \ \ for \ \ m=0.
\end{array}
\end{equation}
With function $u_q(x)=|x|^q, \ \ q<0$ inequality
\eqref{eq23} becomes equality.
\end{itemize}
\end{proposition}
\begin{proof}
For $m\neq0$ the inequality \eqref{eq6} has the form
\begin{equation}
\label{eq12}
\begin{array}{lll}
&&\left(\int_{B_R\backslash B_r}\left|\frac{<x,\nabla
u>}{|x|}\right|^p\right)^{\frac{1}{p}}
\geq\frac{1}{p'}\left(\int_{B_R\backslash
B_r}\frac{|u|^p}{|x|^{(n-1)p'}\left|\frac{R^m-|x|^m}{m}
\right|^p}dx\right)^{\frac{1}{p}}
\\[1pt]
\\
&+&\frac{1}{p}r^{1-n}\left|\frac{R^m-r^m}{m}
\right|^{1-p}\int_{\partial B_r}|u|^pdS\left(\int_{B_R\backslash
B_r}\frac{|u|^p}{|x|^{(n-1)p'}\left|\frac{R^m-|x|^m}{m}
\right|^p}dx\right)^{-\frac{1}{p'}}.
\end{array}
\end{equation}
Analogously, for $m=0$, i. e. $p=n$ the inequality \eqref{eq6} becomes
\begin{equation}
\label{eq13}
\begin{array}{lll}
&&\left(\int_{B_R\backslash B_r}\left|\frac{<x,\nabla
u>}{|x|}\right|^ndx\right)^{\frac{1}{n}}
\geq\frac{n-1}{n}\left(\int_{B_R\backslash
B_r}\frac{|u|^n}{|x|^n\left|\ln\frac{R}{|x|}\right|^n}dx\right)^{\frac{1}{n}}
\\[1pt]
\\
&+&\frac{1}{n}\left(r\ln\frac{R}{r}\right)^{1-n}\int_{\partial
B_r}|u|^ndS\left(\int_{B_R\backslash
B_r}\frac{|u|^n}{|x|^n\left|\ln\frac{R}{|x|}\right|^n}dx\right)^{\frac{1-n}{n}}.
\end{array}
\end{equation}
(i) For $m>0$, after the limit $R\rightarrow\infty$ in \eqref{eq12} we obtain
$$\begin{array}{lll}
&&\frac{1}{p'}\left(\int_{B_R\backslash
B_r}\frac{|u|^p}{|x|^{(n-1)p'}\left|\frac{R^m-|x|^m}{m}
\right|^p}dx\right)^{\frac{1}{p}}\rightarrow_{R\rightarrow\infty}0 \ \ \hbox{and}
\\[1pt]
\\
&&\frac{1}{p}r^{1-n}\left|\frac{R^m-r^m}{m}
\right|^{1-p}\int_{\partial B_r}|u|^pdS\left(\int_{B_R\backslash
B_r}\frac{|u|^p}{|x|^{(n-1)p'}\left|\frac{R^m-|x|^m}{m}
\right|^p}dx\right)^{-\frac{1}{p'}}
\\
&&\rightarrow_{R\rightarrow\infty}\frac{1}{p}r^{1-n}\int_{\partial B_r}|u|^pdS\left(\int_{B_r^c
}\frac{|u|^p}{|x|^{(n-1)p'}}dx\right)^{-\frac{1}{p'}}.
\end{array}
$$
and hence \eqref{eq21} holds.

Let us check that with the function $u_\alpha(x)=e^{-\alpha|x|^m}, \alpha>0$ inequality \eqref{eq21} becomes an equality. Simply computation gives us
$$
\begin{array}{lll}
I_{1\alpha}&=&\left(\int_{
B^c_r}\frac{e^{-\alpha p|x|^m}}{|x|^{(n-1)p'}}dx\right)^{\frac{1}{p'}}
=\left(\int_{S_1}\int_r^\infty e^{-\alpha p\rho^m}\rho^{-(n-1)p'}\rho^{n-1}d\rho dS\right)^{\frac{1}{p'}}
\\
&=&|S_1|^{\frac{1}{p'}}\left(\frac{1}{m}\int_r^\infty e^{-\alpha p\rho^m}d\rho^m\right)^{\frac{1}{p'}}=|S_1|^{\frac{1}{p'}}\frac{1}{(\alpha pm)^{\frac{1}{p'}}}e^{-\alpha (p-1)r^m},
\\[1pt]
\\
I_{2\alpha}&=&\left(\int_{B^c_r}\left|\frac{<x,\nabla
u>}{|x|}\right|^pdx\right)^{\frac{1}{p}}=\left(\int_{S_1}\int_r^\infty e^{-\alpha p\rho^m}(\alpha m)^p\rho^{-(n-1)p'}\rho^{n-1}d\rho dS\right)^{\frac{1}{p}}
\\[1pt]
\\
&=&|S_1|^{\frac{1}{p}}\left(\frac{(\alpha m)^p}{m\alpha p}\int_r^\infty e^{-\alpha p\rho^m}d\rho^m\right)^{\frac{1}{p}}=|S_1|^{\frac{1}{p}}\frac{\alpha m}{(\alpha pm)^{\frac{1}{p}}}e^{-\alpha r^m},
\\[1pt]
\\
I_{3\alpha}&=&\frac{1}{p}r^{1-n}\int_{\partial B_r}|u_\alpha|^pdS=\frac{1}{p}r^{1-n}\int_{S_1}e^{-\alpha pr^m}r^{n-1}dS=|S_1|\frac{1}{p}e^{-\alpha pr^m},
\end{array}
$$
and we get the equality $I_{1\alpha}I_{2\alpha}=I_{3\alpha}$.

(ii) For $m<0$, after the limit $R\rightarrow\infty$ in \eqref{eq12} we obtain \eqref{eq22}. Since $k>\frac{1}{p'}$ it follows that $u_k(x)\in M_1(r,\infty)$. Moreover,
 inequality \eqref{eq22} becomes equality. Indeed,
$$
\begin{array}{lll}
I_{1k}&=&\left(\int_{B^c_r}\left|\frac{<x,\nabla
u_k>}{|x|}\right|^pdx\right)^{\frac{1}{p}}=\left(\int_{B^c_r}\left|km|x|^{km-1}\right|^pdx\right)^{\frac{1}{p}}
\\[1pt]
\\
&=&|km|\left(|S_1|\int_r^\infty\rho^{(km-1)p}\rho^{n-1}d\rho\right)^{\frac{1}{p}}
=|km|\left(|S_1|\frac{r^{(km-1)p+n}}{|(km-1)p+n|}\right)^{\frac{1}{p}},
\\
I_{2k}&=&\frac{|m|}{p'}\left(\int_{
B^c_r}\frac{|u_k|^p}{|x|^{p}}dx\right)^{\frac{1}{p}}=\frac{|m|}{p'}\left(|S_1|\int_r^\infty\rho^{(km-1)p+n-1}d\rho\right)^{\frac{1}{p}}
\\[1pt]
\\
&=&\frac{|m|}{p'}\left(|S_1|\frac{r^{(km-1)p+n}}{|(km-1)p+n|}\right)^{\frac{1}{p}},
\\[1pt]
\\
I_{3k}&=&\frac{1}{p}r^{1-p}\int_{\partial
B_r}|u_k|^pdS\left(\int_{
B^c_r}\frac{|u_k|^p}{|x|^{p}}dx\right)^{-\frac{1}{p'}}
\\
&=&\frac{1}{p}r^{1-p}|S_1|r^{kmp+n-1}\left(|S_1|\frac{r^{(km-1)p+n}}{(km-1)p+n}\right)^{-\frac{1}{p'}}
\\[1pt]
\\
&=&\frac{1}{p}\left(|S_1|r^{|(km-1)p+n|}\right)^{\frac{1}{p}}(|(km-1)p+n|)^{\frac{1}{p'}}.
\end{array}
$$
Since
$$
\frac{|km|}{|(km-1)p+n|}=\frac{|m|}{p'}\frac{1}{|(km-1)p+n|}+\frac{1}{p}
$$
we get the equality $I_{1k}=I_{2k}+I_{3k}$.

(iii) For $m=0$, after the limit $R\rightarrow\infty$ in \eqref{eq13}, since
$$
\begin{array}{lll}
&&\frac{n-1}{n}\left(\int_{B_R\backslash
B_r}\frac{|u|^n}{|x|^n\left|\ln\frac{R}{|x|}\right|^n}dx\right)^{\frac{1}{n}}\rightarrow_{R\rightarrow\infty} 0 \ \ \hbox{and}
\\[1pt]
\\
&&\frac{1}{n}\left(r\ln\frac{R}{r}\right)^{1-n}\int_{\partial
B_r}|u|^ndS\left(\int_{B_R\backslash
B_r}\frac{|u|^n}{|x|^n\left|\ln\frac{R}{|x|}\right|^n}dx\right)^{\frac{1-n}{n}}
\\[1pt]
\\
&&=\frac{1}{n}r^{1-n}\ln^{1-n}R\left(1-\frac{\ln r}{\ln R}\right)^{1-n}\int_{\partial
B_r}|u|^ndS\frac{1}{\ln^{1-n}R}\left(\int_{B_R\backslash
B_r}\frac{|u|^n}{|x|^n\left|1-\frac{\ln |x|}{\ln R}\right|^n}dx\right)^{\frac{1-n}{n}}
\\[1pt]
\\
&&\rightarrow_{R\rightarrow\infty}\frac{1}{n}r^{1-n}\int_{\partial
B_r}|u|^ndS\left(\int_{B_r^c}\frac{|u|^n}{|x|^n}dx\right)^{\frac{1-n}{n}}
\end{array}
$$
we obtain \eqref{eq23}.
Let us check that with function $u_q(x)=|x|^{q}, q<0$ inequality \eqref{eq23} becomes equality. Indeed,
$$
\begin{array}{lll}
I_{1q}&=&\left(\int_{
B^c_r}\frac{|u_q|^n}{|x|^{n}}dx\right)^{\frac{n-1}{n}}
=\left(|S_1|\int_r^\infty\rho^{(q-1)n}\rho^{n-1}d\rho\right)^{\frac{n-1}{n}}=|S_1|^{\frac{n-1}{n}}(n|q|)^{-\frac{n-1}{n}}r^{(n-1)q},
\\[1pt]
\\
I_{2q}&=&\left(\int_{B^c_r}\left|\frac{<x,\nabla
u_q>}{|x|}\right|^ndx\right)^{\frac{1}{n}}=\left(|S_1|\int_r^\infty |q|^n\rho^{(q-1)n}\rho^{n-1}d\rho\right)^{\frac{1}{n}}=|S_1|^{\frac{1}{n}}|q|^{\frac{n-1}{n}}n^{-\frac{1}{n}}r^q,
\\[1pt]
\\
I_{3q}&=&\frac{1}{n}r^{1-n}\int_{\partial
B_r}|u_q|^ndS=\frac{1}{n}r^{1-n}|S_1|r^{nq}r^{n-1}=|S_1|\frac{1}{n}r^{qn},
\end{array}
$$
and we get the equality $I_{1q}I_{2q}=I_{3q}$.
\end{proof}
\begin{proposition}
\label{prop2}
For every $u\in M_2(r,\infty)$ the following inequalities hold:
\begin{itemize}
\item[(i)]
\begin{equation}
\label{eq24}
\begin{array}{lll}
&&\left(\int_{B^c_r}\left|\frac{<x,\nabla
u>}{|x|}\right|^pdx\right)^{\frac{1}{p}}
\geq\frac{m}{p'}\left(\int_{
B^c_r}\frac{|u|^p}{|x|^{(n-1)p'}\left|r^m-|x|^m\right|^p}dx\right)^{\frac{1}{p}}
\\[1pt]
\\
&+&\frac{1}{p}\hbox{limsup}_{R\rightarrow\infty}\left[R^{1-p}\int_{\partial
B_R}|u|^pdS\right]\left(\int_{
B^c_r}\frac{|u|^p}{|x|^{(n-1)p'}\left||x|^m-r^m\right|^p}dx\right)^{-\frac{1}{p'}}, 
\\[1pt]
\\ 
&& for \ \ m>0.
\end{array}
\end{equation}
With function $u_\varepsilon(x)=|x|^{-\frac{m(1-\varepsilon)}{p'}}\left(|x|^m-r^m\right)^{\frac{(1+\varepsilon)}{p'}}$, $0<\varepsilon<1$ inequality
\eqref{eq24} is $\varepsilon$-sharp, i.e.
\begin{equation}
\label{eq180}
\left(\frac{m}{p'}\right)^p\leq\frac{L^1(u_{\varepsilon})}{\left[\left(K^1(u_{\varepsilon})\right)^{\frac{1}{p}}
+K^1_0(u_{\varepsilon})\left(K^1(u_{\varepsilon})\right)^{-\frac{1}{p'}}\right]^p}\leq\frac{L^1(u_{\varepsilon})}{K^1(u_{\varepsilon})}
\leq\left(\frac{m}{p'}\right)^p(1+\varepsilon)^p.
\end{equation}
\item[(ii)]
\begin{equation}
\label{eq25}
\begin{array}{lll}
&&\left(\int_{B^c_r}\left|\frac{<x,\nabla
u>}{|x|}\right|^pdx\right)^{\displaystyle\frac{1}{p}}
\geq\frac{|m|}{p'}\left(\int_{
B^c_r}\frac{|u|^p}{|x|^{(n-1)p'}\left|r^m-|x|^m\right|^p}dx\right)^{\frac{1}{p}}
\\[1pt]
\\
&+&\frac{r^{n-p}}{p}\hbox{limsup}_{R\rightarrow\infty}\left[R^{1-n}\int_{\partial
B_R}|u|^pdS\right]\left(\int_{
B^c_r}\frac{|u|^p}{|x|^{(n-1)p'}\left||x|^m-r^m\right|^p}dx\right)^{-\frac{1}{p'}}, 
\\ 
&& for \ \ m<0.
\end{array}
\end{equation}
With function $u_s(x)=\left(r^m-|x|^m\right)^{s}, \ \
s>\frac{1}{p'}$ inequality
\eqref{eq25} becomes equality.
\item[(iii)]
\begin{equation}
\label{eq26}
\begin{array}{lll}
&&\left(\displaystyle\int_{B^c_r}\left|\frac{<x,\nabla
u>}{|x|}\right|^ndx\right)^{\displaystyle\frac{1}{n}}
\geq\frac{n-1}{n}\left(\int_{B^c_r}\frac{|u|^n}{|x|^n
\left|\ln\displaystyle\frac{|x|}{r}\right|^n}dx\right)^{\displaystyle\frac{1}{n}}
\\[1pt]
\\
&+&\displaystyle\frac{1}{n}\hbox{limsup}_{R\rightarrow\infty}
\left[\left(R\ln\displaystyle\frac{R}{r}\right)^{1-n}\int_{\partial
B_R}|u|^ndS\right]\left(\int_{
B^c_r}\frac{|u|^n}{|x|^n\left|\ln\displaystyle\frac{|x|}{r}
\right|^n}dx\right)^{\displaystyle\frac{1-n}{n}}, 
\\ 
&& for \ \ m=0.
\end{array}
\end{equation}
For function
$$
u_\eta(x)=\left\{\begin{array}{l} \left(\ln\frac{|x|}{r}\right)^{\frac{n-1}{n}(1+\frac{\eta}{2})}M^{\frac{n-1}{n}(1-\frac{\eta}{2})}, \ \ r<|x|<M
\\[1pt]
\\
\left(\ln\frac{|x|}{r}\right)^{\frac{n-1}{n}(1+\frac{\eta}{2})}|x|^{\frac{n-1}{n}(1-\frac{\eta}{2})}, \ \ M<|x| \end{array}\right.,
$$
where $0<\eta<1$, inequality   \eqref{eq26} is $\eta$-sharp, i.e.
\begin{equation}
\label{eq200}
\left(\frac{n-1}{n}\right)^n\leq\frac{L^1(u_{\eta})}{K^1(u_{\eta})} \leq
\left(\frac{n-1}{n}\right)^n(1+\eta)^n.
\end{equation}
\end{itemize}
\end{proposition}
\begin{proof}
For $m\neq0$ inequality \eqref{eq6} has the form
\begin{equation}
\label{eq220}
\begin{array}{lll}
&&\left(\int_{B_R\backslash B_r}\left|\frac{<x,\nabla
u>}{|x|}\right|^p\right)^{\frac{1}{p}}
\geq\frac{|m|}{p'}\left(\int_{B_R\backslash
B_r}\frac{|u|^p}{|x|^{(n-1)p'}\left|r^m-|x|^m
\right|^p}dx\right)^{\frac{1}{p}}
\\[1pt]
\\
&+&\frac{1}{p}R^{1-n}\left|R^m-r^m
\right|^{1-p}\int_{\partial B_R}|u|^pdS\left(\int_{B_R\backslash
B_r}\frac{|u|^p}{|x|^{(n-1)p'}\left|r^m-|x|^m
\right|^p}dx\right)^{-\frac{1}{p'}}.
\end{array}
\end{equation}
while for $m=0$ \eqref{eq6} becomes
\begin{equation}
\label{eq221}
\begin{array}{lll}
&&\left(\int_{B_R\backslash B_r}\left|\frac{<x,\nabla
u>}{|x|}\right|^ndx\right)^{\frac{1}{n}}
\geq\frac{n-1}{n}\left(\int_{B_R\backslash
B_r}\frac{|u|^n}{|x|^n\left|\ln\frac{|x|}{r}\right|^n}dx\right)^{\frac{1}{n}}
\\[1pt]
\\
&+&\frac{1}{n}\left(R\ln\frac{R}{r}\right)^{1-n}\int_{\partial
B_r}|u|^ndS\left(\int_{B_R\backslash
B_r}\frac{|u|^n}{|x|^n\left|\ln\frac{|x|}{r}\right|^n}dx\right)^{\frac{-1}{n'}}.
\end{array}
\end{equation}
(i) If $m>0$ after the limit $R\rightarrow\infty$ in \eqref{eq220} we obtain
$$
\begin{array}{lll}
&&\frac{1}{p}\hbox{limsup}_{R\rightarrow\infty}\left[R^{1-n}\left|R^m-r^m\right|^{1-p}\int_{\partial
B_R}|u|^pdS\right]
\\
&&=\frac{1}{p}\hbox{limsup}_{R\rightarrow\infty}\left[R^{1-n}R^{n-p}\left|1-\frac{r^m}{R^m}\right|^{1-p}\int_{\partial
B_R}|u|^pdS\right]
\\[1pt]
\\
&&=\frac{1}{p}\hbox{limsup}_{R\rightarrow\infty}R^{1-p}\int_{\partial
B_R}|u|^pdS
\end{array}
$$
which proves \eqref{eq24}.

For the function $u_\varepsilon(x)=|x|^{-\frac{m(1-\varepsilon)}{p'}}\left(|x|^m-r^m\right)^{\frac{(1+\varepsilon)}{p'}}$, $0<\varepsilon<1$, $u_\varepsilon(x)\in M_2(r,\infty)$ simple computation give us
$$
\begin{array}{lll}
&&I_{1\varepsilon}=L^1(u_\varepsilon)=\int_{B^c_r}\left|\frac{<x,\nabla
u_\varepsilon>}{|x|}\right|^pdx
\\[1pt]
\\
&&=\left(\frac{m}{p'}\right)^p\int_{B^c_r}|x|^{-m(1-\varepsilon)(p-1)-p}\left(|x|^m-r^m\right)^{(1+\varepsilon(p-1)-p}
\\[1pt]
\\
&&\times\left[(1+\varepsilon)|x|^m-(1-\varepsilon)\left(|x|^m-r^m\right)\right]^pdx
\\[1pt]
\\
&&=\left(\frac{m}{p'}\right)^p|S_1|\int_r^\infty(\rho^m-r^m)^{(1+\varepsilon)(p-1)-p}
\\[1pt]
\\
&&\times\left[(1+\varepsilon)\rho^m-(1-\varepsilon)(\rho^m-r^m)\right]^p\rho^{-m(1-\varepsilon)(p-1)-p+n-1}d\rho
\\[1pt]
\\
&&\leq\left(\frac{m}{p'}\right)^p|S_1|(1+\varepsilon)^p\int_r^\infty(\rho^m-r^m)^{(1+\varepsilon)(p-1)-p}\rho^{-m(1-\varepsilon)(p-1)-p+n-1+mp}d\rho
\\[1pt]
\\
&&=\left(\frac{m}{p'}\right)^p|S_1|(1-\varepsilon)^p\int_r^\infty(\rho^m-r^m)^{(1+\varepsilon)(p-1)-p}\rho^{-m(1+\varepsilon)(p-1)-\frac{n-1}{p-1}}d\rho,
\end{array}
$$
because $n-1-p+mp=-\frac{n-1}{p-1}$.
$$
\begin{array}{lll}
I_{2\varepsilon}&=&K^1(u_\varepsilon)=\int_{B^c_r}|x|^{-m(1-\varepsilon)(p-1)-(n-1)p'}\left(|x|^m-r^m\right)^{(1+\varepsilon)(p-1)-p}dx
\\[1pt]
\\
&=&|S_1|\int_r^\infty\left(\rho^m-r^m\right)^{(1+\varepsilon)(p-1)-p}\rho^{-m(1-\varepsilon)(p-1)-\frac{n-1}{p-1}}d\rho.
\end{array}
$$
Thus \eqref{eq180} follows immediately from expressions of $I_{1\varepsilon}$ and $I_{2\varepsilon}$.
(ii) When $m<0$, after the limit $R\rightarrow\infty$ in \eqref{eq220} we get
$$
\begin{array}{lll}
&&\frac{1}{p}\hbox{limsup}_{R\rightarrow\infty}\left[R^{1-n}\left|R^m-r^m\right|^{1-p}\int_{\partial
B_R}|u|^pdS\right]
\\[1pt]
\\
&&=\frac{1}{p}\hbox{limsup}_{R\rightarrow\infty}\left[r^{n-p}R^{1-n}\left|\frac{R^m}{r^m}-1\right|^{1-p}\int_{\partial
B_R}|u|^pdS\right]
\\[1pt]
\\
&&=\frac{r^{n-p}}{p}\hbox{limsup}_{R\rightarrow\infty}R^{1-n}\int_{\partial
B_R}|u|^pdS
\end{array}
$$
which proves \eqref{eq25}.

For the function $u_s(x)=\left(r^m-|x|^m\right)^{s}, \ \
s>\frac{1}{p'}$  we have the identities
$$
\begin{array}{lll}
I_{1s}&=&\left(\int_{B^c_r}\left|\frac{<x,\nabla
u_s>}{|x|}\right|^pdx\right)^{\frac{1}{p}}=\left(s^p|m|^p|S_1|\int_r^\infty\rho^{p(m-1)+n-1}\left(r^m-\rho^m\right)^{(s-1)p}d\rho\right)^{\frac{1}{p}}
\\[1pt]
\\
&=&s|m||S_1|^{\frac{1}{p}}\left(-\frac{1}{m}\int_r^\infty\left(r^m-\rho^m\right)^{(s-1)p}d\left(r^m-\rho^m\right)\right)^{\frac{1}{p}},
\\[1pt]
\\
&=&s|m|^{\frac{1}{p'}}|S_1|^{\frac{1}{p}}r^{\frac{m}{p}[(s-1)p+1]}[(s-1)p+1]^{-\frac{1}{p}}
\end{array}
$$
because $p(m-1)+n-1=m-1$.
$$
\begin{array}{lll}
I_{2s}&=&\int_{
B^c_r}\frac{|u_s|^p}{|x|^{(n-1)p'}\left|r^m-|x|^m\right|^p}dx=|S_1|\int_r^\infty\rho^{(n-1)(1-p')}\left|r^m-\rho^m\right|^{(s-1)p}d\rho
\\[1pt]
\\
&=&\frac{|S_1|}{|m|}\frac{r^{m[(s-1)p+1]}}{(s-1)p+1},
\\[1pt]
\\
I_{3s}&=&\frac{r^{n-p}}{p}\hbox{limsup}_{R\rightarrow\infty}\left[R^{1-n}\int_{\partial
B_R}\left|r^m-R^m\right|^{sp}dS\right]
\\[1pt]
\\
&=&\frac{r^{n-p}}{p}\hbox{limsup}_{R\rightarrow\infty}\left|R^m-r^m\right|^{sp}|S_1|=\frac{|S_1|}{p}r^{n-p+mps}.
\end{array}
$$
Simple computation gives us the equality
$$
\begin{array}{lll}
&&\frac{|m|}{p'}\left(I_{2s}\right)^{\frac{1}{p}}+I_{3s}\left(I_{2s}\right)^{-\frac{1}{p'}}
\\[1pt]
\\
&=&\frac{|S_1||m|^{\frac{1}{p'}}r^{\frac{m}{p}[(s-1)p+1]}}{p'[(s-1)p+1]^{\frac{1}{p}}}
+\frac{|S_1|r^{n-p+mps}}{p}\left(\frac{|S_1|}{|m|}\right)^{-\frac{1}{p'}}\left[\frac{r^{m[(s+1)p+1]}}{(s-1)p+1}\right]^{-\frac{1}{p'}}
\\[1pt]
\\
&=&\frac{1}{p}|m|^\frac{1}{p'}|S_1|^{\frac{1}{p}}r^{\frac{m}{p}[(s-1)p+1]}\left(\frac{p-1}{[(s-1)p+1]^p}+[(s-1)p+1]^{\frac{1}{p'}}\right)
\\[1pt]
\\
&=&s|m|^\frac{1}{p'}|S_1|^{\frac{1}{p}}r^{\frac{m}{p}[(s-1)p+1]}[(s-1)p+1]^{-\frac{1}{p}}=I_{1s}
\end{array}
$$
because
$$
n-p+mps-\frac{m}{p'}[(s-1)p+1]=\frac{m}{p}[(s-1)p+1]
$$
(iii) After the limit $R\rightarrow\infty$ in \eqref{eq221} we get \eqref{eq26}.

The function $u_\eta(x)$ belongs to $M_2(r,\infty)$ for $m=0$. Moreover, for $r<|x|<M$ we have the equalities
$$
\begin{array}{lll}
&&I_{1\eta}=L^1(u_\eta)=M^{(n-1)(1-\eta)}\left(1+\frac{\eta}{2}\right)^n\left(\frac{n-1}{n}\right)^n|S_1|\int_r^M\left(\ln\frac{\rho}{r}\right)^{(n-1)\left(1+\frac{\eta}{2}\right)-n}\rho^{-1}d\rho,
\\[1pt]
\\
&&I_{2\eta}=K^1(u_\eta)=M^{(n-1)(1+\frac{\eta}{2})}|S_1|\int_r^M\left(\ln\frac{\rho}{r}\right)^{(n-1)\left(1+\frac{\eta}{2}\right)-n}\rho^{-1}d\rho,
\end{array}
$$
and hance
$$
\frac{L^1(u_\eta)}{K^1(u_\eta)}=\left(\frac{n-1}{n}\right)^n\left(1+\frac{\eta}{2}\right)^n\leq\left(\frac{n-1}{n}\right)^n\left(1+\eta\right)^n.
$$
Tedious calculations give us for $|x|\geq M$ the identities
$$
\begin{array}{lll}
&&I_{1\eta}=L^1(u_\eta)=\left(\frac{n-1}{n}\right)^n|S_1| \int_M^\infty\rho^{(n-1)(1-\frac{\eta}{2})-1}\left(\ln\frac{\rho}{r}\right)^{(n-1)\left(1+\frac{\eta}{2}\right)-n}
\\[1pt]
\\
&&\times\left[\left(1-\frac{\eta}{2}\right)\frac{1}{\rho}\ln\frac{\rho}{r}+\left(1+\frac{\eta}{2}\right)\right]^nd\rho
\\[1pt]
\\
&&I_{2\eta}=K^1(u_\eta)=|S_1|\int_r^M\rho^{(n-1)(1-\frac{\eta}{2}-1}\left(\ln\frac{\rho}{r}\right)^{(n-1)\left(1+\frac{\eta}{2}\right)-n},
\end{array}
$$
If $M$ is sufficiently large, i.e.,
$$
\frac{1}{M}\ln\frac{M}{r}<\frac{\eta}{2-\eta} \ \ \hbox{ and } \ \ M>er,
$$
then
$$
\left(1-\frac{\eta}{2}\right)\frac{1}{\rho}\ln\frac{\rho}{r}+1+\frac{\eta}{2}\leq\left(1-\frac{\eta}{2}\right)\frac{\eta}{2-\eta}+1+\frac{\eta}{2}=1+\eta,
$$
because the function $h(\rho)=\frac{1}{\rho}\ln\frac{\rho}{r}$ is monotone decreasing for $\rho>er$. Thus \eqref{eq200} follows from expressions of $I_{1\eta}$ and $I_{2\eta}$ above for $|x|\geq r$.
\end{proof}
\section{Sharp inequalities in an exterior of a ball   $B^c_r$ for $u\in W_0^{1,p}(B_r^c)$}
\label{sec4}
For $m<0$, i.e. $p<n$, we can combine inequalities
\eqref{eq12} and \eqref{eq22} for functions $u\in M(r,\infty)$, where

$$
M(r,\infty)=\left\{\begin{array}{l} u\in W^{1,p}_0(B^c_r),
\\[1pt]
\\
\left(\displaystyle\frac{\hat{r}^m-r^m}{m}\right)^{1-p}\displaystyle\int_{\partial
B_{\hat{r}}}|u|^pdx\rightarrow0, \ \ \hat{r}\rightarrow r+0,
\\[1pt]
\\
R^{1-p}\displaystyle\int_{\partial
B_{R}}|u|^pdx\rightarrow0, \ \ \hbox{ for }
R\rightarrow\infty.
\end{array}\right.
$$
For  $r<\gamma<\infty$, $\gamma=2^{\frac{1}{|m|}}r$, we define
$$
\begin{array}{l}
L_1(u)=\int_{B_\gamma\backslash B_r}\left|\frac{\langle x,\nabla
u\rangle}{|x|}\right|^pdx,\ \ L_2(u)=\int_{B^c_\gamma}\left|\frac{\langle x,\nabla
u\rangle}{|x|}\right|^pdx,
\\[1pt]
\\
K_{11}(u)=|m|^p\int_{B_{\gamma}\backslash B_r}\frac{|u|^p}{|x|^{(n-1)p'}(|x|^m-r^m)^p}dx, \ \ K_{12}(u)=|m|^p\int_{B^c_{\gamma}}\frac{|u|^p}{|x|^p}dx,
\\[1pt]
\\
K_{01}(u)=|m|^{p-1}\gamma^{1-p}\int_{\partial
B_{\gamma}}|u|^pd\sigma=
K_{02}(u)=|m|^{p-1}\gamma^{1-p}\int_{\partial B_{\gamma}}|u|^pd\sigma=K_0.
\end{array}
$$
\begin{proposition}
\label{prop3}
If $\gamma=2^{1/|m|}r$ then  for every $u\in M(r,\infty)$ the inequality
\begin{equation}
\label{eq304}
L(u)=\sum_1^2L_j(u)=\int_{B_r^c}\left|\frac{\langle x,\nabla u\rangle}{|x|}\right|^pdx\geq\left(\frac{1}{p}\right)^p\sum_1^2\frac{[K_{0j}(u)+(p-1)
K_{1j}(u)]^p}{(K_{1j}(u))^{p-1}}=K(u), 
\end{equation}
holds in $B^c_r$.

The inequality \eqref{eq304} becomes an equality for functions $u_\beta(x)\in M(r,\infty)$, $\beta>1/p'$ where
$$
u_\beta(x)=\left\{\begin{array}{l}
(r^m-|x|^m)^{\beta} \ \  x\in
B_{\gamma}\backslash B_r,
\\[1pt]
\\
|x|^{m\beta},  \ \  x\in B^c_{\gamma}.
\end{array}\right.
$$
\end{proposition}
\begin{proof}
As in Theorem \ref{th1} and Proposition \ref{prop1} ii) we
get
\begin{equation}
\label{eq302}\begin{array}{l}
L_1(u)\geq K_1(u) \ \ \hbox{ in }\ \  B_{\gamma}\backslash B_r, 
\\[1pt]
\\
L_2(u)\geq K_2(u) \ \ \hbox{ in } \ \ B^c_{\gamma}
\end{array}
\end{equation}
where
$$
K_j(u)=\frac{\left[K_{0j}(u)+(p-1)K_{1j}(u)\right]^p}{\left(K_{1j}(u)\right)^{p-1}}.
$$
With the choice of  $\gamma=2^{1/|m|}r$, so $\gamma^p=\gamma^{(n-1)p'}(r^m-\gamma^m)^p$, we have  continuous kernel for $K_{j1}$ on $\partial
B_{\gamma}$.

Adding inequalities in
\eqref{eq302} we get for $u\in M(r,\infty)$ $L(u)=\sum_1^2L_j(u)\geq\sum_1^2K_j(u)$.
Hance  from the Joung inequality  we obtain \eqref{eq304}

Using \eqref{eq22} in $B_\gamma^c$ and \eqref{eq13} in $B_{\gamma\backslash r}$ it follows that the inequality \eqref{eq304} becomes equality for functions $u_\beta(x), \beta>\frac{1}{p'}$.
\end{proof}
Proposition \ref{prop3} gives   sharp Hardy inequality in an exterior of a
ball $B^c_r$ for functions $u\in W^{1,p}_0(B^c_r)$, $p<n$.

We will illustrate Proposition \ref{prop2} ii) and Proposition \ref{prop3} in the following examples.
\begin{example}
\label{ex1}\rm
For $p=2$, $n\geq3$, $r=1$, $a=0$ and $m=2-n<0$ from Proposition \ref{prop1} ii)  it follows that the Hardy  inequality
\begin{equation}
\label{eq02}
\int_{B_1^c}|\nabla u|^2dx\geq\left(\frac{n-2}{2}\right)^2\int_{B_1^c}\frac{|u|^2}{|x|^2}dx+\frac{n-2}{2}\int_{\partial B_1^c}u^2+\frac{1}{4}\int_{\partial B_1^c}u^2\left( \int_{B_1^c}\frac{|u|^2}{|x|^2}\right)^{-1},
\end{equation}
holds. Note that \eqref{eq02} has an additional term in the right hand side in comparison with  \eqref{eq01}. Moreover, for the functions $u_k(x)=|x|^{k(2-n)}, k>2$  inequality \eqref{eq02} becomes an equality, i.e., inequality \eqref{eq02} is sharp.
\end{example}
\begin{example}
\label{ex2} \rm
For $p=2, n\geq3, m=2-n<0$, $\gamma=2^{\frac{1}{n-2}}r$,  and every function $u\in W_0^{1,2}(B_r^c)$ the Hardy inequality \eqref{eq302} becomes
\begin{equation}
\label{eq05}
\begin{array}{lll}
&&\int_{B_r^c}\left|\frac{\langle x,\nabla u\rangle}{|x|}\right|^pdx\geq\left(\frac{n-2}{2}\right)^2\int_{B_r^c}\frac{u^2}{|x|^2}dx
\\[1pt]
\\
&&+2\left(\frac{n-2}{2}\right)^2\int_{B_\gamma\backslash B_r}\frac{\left|\frac{x}{r}\right|^{n-2}\left(1-\left|\frac{x}{\gamma}\right|^{n-2}\right)}{\left(\left|\frac{x}{r}\right|^{n-2}-1\right)^2}u^2dx
+\frac{2^{-\frac{1}{n-2}}}{r}(n-2)\int_{\partial B_\gamma}u^2dS
\\[1pt]
\\
&&+\frac{2^{-\frac{2}{n-2}}}{r^2}\left(\frac{4}{n-2}\right)^2\left(\int_{\partial B_\gamma}u^2dS\right)^2
\\[1pt]
\\
&&\times\left[\left(\int_{B_\gamma\backslash B_r}\frac{u^2}{|x|^2\left(\left|\frac{x}{r}\right|^{n-2}-1\right)^2}dx\right)^{-2}+\left(\int_{B_\gamma^c}\frac{u^2}{|x|^2}dx\right)^{-2}\right]
\end{array}
\end{equation}
Moreover, for function $u_\beta(x)$ defined in Proposition \ref{prop3} for $m=2-n$, inequality \eqref{eq05} becomes equality.

Let us mention that Hardy inequality \eqref{eq05} has the same leading term in the right hand side as in inequality \eqref{eq03}, but inequality \eqref{eq05} is sharp one.

Finally, it is difficult to compare inequality \eqref{eq04} with \eqref{eq05}, but \eqref{eq05} is sharp one, i.e., for the functions $u_\beta(x)$ defined in Proposition \ref{prop3}, inequality \eqref{eq05} becomes an equality.
\end{example}

\textbf{Acknowledgement}
\small{This paper is  partially supported by the National Scientific Program "Information and Communication Technologies for a Single Digital Market in Science, Education and Security (ICTinSES)", contract No DO1–205/23.11.2018, financed by the Ministry of Education and Science in Bulgaria and also  by the Grant No BG05M2OP001--1.001--0003, financed by the Science and Education for Smart Growth Operational Program (2014-2020) in Bulgaria and co-financed by the European Union through the European Structural and Investment Funds.}

\end{document}